\documentclass[12pt]{amsart}

\author{Paul \textsc{Poncet}}
\address{CMAP, \'{E}cole Polytechnique, Route de Saclay, 91128 Palaiseau Cedex, France \\
and INRIA, Saclay--\^{I}le-de-France}

\email{poncet@cmap.polytechnique.fr}

\usepackage{stmaryrd}
\usepackage{amssymb}
\usepackage{amsmath}
\usepackage{graphicx}
\usepackage{t1enc}
\usepackage[latin1]{inputenc}

\usepackage{yhmath}
\usepackage{mathabx}

\usepackage{times} 
\usepackage{bbm}
\usepackage{ifpdf}
\def\twoheaddownarrow{\rlap{$\downarrow$}\raise-.5ex\hbox{$\downarrow$}}
\def\twoheaduparrow{\rlap{$\uparrow$}\raise.5ex\hbox{$\uparrow$}}

\newtheorem{theorem}{Theorem}[section]
\newtheorem{corollary}[theorem]{Corollary}
\newtheorem{proposition}[theorem]{Proposition}
\newtheorem{lemma}[theorem]{Lemma}

\theoremstyle{definition}

\newtheorem{remark}[theorem]{Remark}

\begin{document}

\title{A memo on chains and their topologies \\ (work in progress)}

\date{\today}



\begin{abstract}	
We summarize some facts on chains (totally ordered sets), from an order-theoretic and from a topological point of view. We highlight the fact that many classical theorems that are true for partially ordered sets under some completeness assumption remain true for chains without any kind of completeness. 
\end{abstract}

\maketitle







\section{Every chain is a continuous poset}

A \textit{partially ordered set} or \textit{poset} $(P,\leqslant)$ is a set $P$ equipped with a reflexive, transitive, and antisymmetric binary relation $\leqslant$. 
If $P$ has a least (resp.\ greatest) element, we customarily denote it by $0$ (resp.\ by $\infty$). 
A poset is \textit{complete} (resp.\ \textit{conditionally-complete}) if every subset (resp.\ every nonempty subset bounded above) has a supremum, in which case every subset (resp.\ every nonempty lower-bounded subset) has an infimum. A poset is \textit{up-complete} if every nonempty subset has a supremum. 

A nonempty subset $C$ of $P$ is a \textit{chain} (or a \textit{totally ordered subset}) if, for all $x, y \in C$, one of the relations $x \leqslant y$ and $y \leqslant x$ holds. We write $x < y$ when $x \leqslant y$ and $x \neq y$. 
A chain $M$ is \textit{maximal} if, for each chain $C$ in $P$, $C \supset M$ implies $C = M$. 

A nonempty subset $D$ of a poset $P$ is \textit{directed} if, for all $x, y \in D$, one can find $z \in D$ such that $x \leqslant z$ and $y \leqslant z$. 
We say that $x \in P$ is \textit{way-below} $y \in P$, written $x \ll y$, if, for every directed subset $D$ with a supremum $\bigvee D$, $y \leqslant \bigvee D$ implies $x \leqslant d$ for some $d \in D$. 
We then say that the poset $P$ is \textit{continuous} if $\twoheaddownarrow x := \{ y \in P : y \ll x \}$ is directed and $x = \bigvee \twoheaddownarrow x$, for all $x \in P$. 
The next result describes the way-below relation on a chain. 

\begin{lemma}\label{lem:impl}
Let $(P,\leqslant)$ be a chain. The following implications hold: 
$$
x < y \Rightarrow x \ll y \Rightarrow x \leqslant y, 
$$ 
for all $x, y \in P$. 
\end{lemma}

\begin{proof}
Assume that $x < y$ and that $y \leqslant \bigvee B$, for some $B \subset P$ with supremum. If $b < x$ for all $b \in B$, then $\bigvee B \leqslant x < y \leqslant \bigvee B$, which is not. Thus, $x \leqslant b$ for some $b \in B$, and $x \ll y$.
\end{proof}

We write $\Downarrow\!\! x$ for $P \ \setminus \uparrow\!\! x = \{ y \in P : y < x \}$, and $\Uparrow\!\! x$ dually. 
Recall that an element $x$ of a poset $P$ is \textit{compact} if $x \ll x$. 

In \cite[Example~I-1.7]{Gierz03}, we read that every complete chain is a continuous lattice, but the following more specific result has not been asserted before Zhao and Zhou \cite{Zhao06}. 


\begin{theorem}
Let $(P,\leqslant)$ be a chain. Every point $x$ of $P$ is compact or satisfies $x = \bigvee \Downarrow\!\! x$. In particular, $P$ is a completely distributive (or supercontinuous) poset, hence also a continuous poset.  
\end{theorem}

\begin{proof}
Suppose that $x$ is not compact. This means that there exists some $B_0 \subset \Downarrow\!\! x$ with supremum such that $x = \bigvee B_0$. Let $u$ be an upper bound of $\Downarrow\!\! x$. Since $B_0 \subset \Downarrow\!\! x$, $u$ is also an upper bound of $B_0$, so that $x = \bigvee B_0 \leqslant u$. This proves that $x$ is the supremum of $\Downarrow\!\! x$. 
\end{proof}

A transitive binary relation $\prec$ on a set is \textit{idempotent} if, for all $x, y$ such that $x \prec y$, there exists some $z$ with $x \prec z \prec y$. 
A poset is \textit{order-dense} if the strict order relation $<$ is idempotent. 
Remembering that the way-below relation of a continuous poset is idempotent, the following result is a clear consequence of the previous theorem. 

\begin{corollary}\label{coro:eq}
Let $(P,\leqslant)$ be a chain. The following conditions are equivalent:
\begin{enumerate}
	\item\label{coro:eq2} the relations $<$ and $\ll$ agree (except perhaps at $0$), 
	\item\label{coro:eq3} $P$ has no compact element (except perhaps $0$), 
\end{enumerate}
If these conditions are satisfied, then $P$ is order-dense. Conversely, if $P$ is conditionally-complete and order-dense, then these conditions are satisfied. 
\end{corollary}

\begin{proof}
Condition (\ref{coro:eq2}) clearly implies (\ref{coro:eq3}), and the fact that (\ref{coro:eq3}) implies (\ref{coro:eq2}) is a consequence of Lemma~\ref{lem:impl}. 
By continuity of $P$, $\ll$ is idempotent, hence $P$ is order-dense if (\ref{coro:eq2}) is satisfied. 
Assume now that $P$ is conditionally-complete and order-dense. If $x$ is a compact element and $\Downarrow\!\! x \neq \emptyset$, then $\Downarrow\!\! x$ has a supremum $u$. Since $x$ is compact, we cannot have $u = x$, so that $u < x$. But $P$ is order-dense, hence there exists some $v$ with $u < v < x$, which contradicts the definition of $u$. Thus, we get $\Downarrow\!\! x = \emptyset$, i.e.\ $x$ is the least element of $P$. 
\end{proof}



\section{Topologies and convergence on chains}

\subsection{Topologies on chains}

Let $(P,\leqslant)$ be a chain. The \textit{upper topology}, denoted by $\nu(P)$, is the topology generated by the complements $\Uparrow\!\! x$ of principal ideals (as subbasic open sets). Dually, the \textit{lower topology} $\omega(P)$ is the topology generated by the complements $\Downarrow\!\! x$ of principal filters. The \textit{intrinsic topology} $i(P)$ is the join of the lower and the upper topologies. 

\begin{proposition}
Let $(P, \leqslant)$ be a chain equipped with its intrinsic topology. Then $P$ is a pospace (the order $\leqslant$ is closed in $P \times P$), hence Hausdorff, and is a topological lattice. 
\end{proposition}

\begin{proof}
The aim is to show that the map $f : (x, y) \mapsto x\wedge y$ is continuous. Let $G$ be an open subset of $P$ such that $x \wedge y \in G$ for some $x, y \in G$ with $x < y$. Then $(x, y) \in G \times \Uparrow\!\! x$ and $G \wedge \Uparrow\!\! x \subset G$, so we deduce that $f^{-1}(G)$ is open in $P$. Dually, the map $(x, y) \mapsto x\vee y$ is continuous, hence $P$ is a topological lattice. 
\end{proof}

In a poset $P$, we write $B^{\uparrow}$ (resp.\ $B^{\downarrow}$) for the subset made up of upper bounds (resp.\ lower bounds) of the subset $B$.  
It is remarkable that the following proposition holds even for non-complete chains. 

\begin{proposition}
Let $(P,\leqslant)$ be a chain. Then, 
\begin{enumerate}
	\item the upper topology and the Scott topology agree on $P$, 
	\item the lower topology and the dual Scott topology agree on $P$, 
	\item the intrinsic topology, the interval topology, the open-interval topology, 
	the order topology, the bi-Scott topology, the Lawson topology, and the dual Lawson topology agree on $P$. 
\end{enumerate}
\end{proposition}

\begin{proof}
It is well known that every principal ideal is Scott-closed, hence the Scott topology refines the upper topology. Now let $F$ be a Scott-closed subset of $P$. We want to show that $F = \bigcap_{x \in F^{\uparrow}} \downarrow\!\! x$, the latter being closed in the upper topology. The inclusion $\subset$ is clear. So let $y \in \bigcap_{x \in F^{\uparrow}} \downarrow\!\! x$, which means that $y$ is lower than every upper bound of $F$. Assume that $y \notin F$. If $f \in F$, we cannot have $y \leqslant f$ since $F = \downarrow\!\! F$. Consequently, $y > f$, and we deduce that $y$ is an upper bound of $F$, and even the least upper bound: $y = \bigvee F$. Since $F$ is Scott-closed and admits $F$ as a directed subset with supremum, we have $\bigvee F = y \in F$, a contradiction. So $y \in F$, and $F$ is closed in the upper topology. 

Using the equality between the upper and the Scott topology and \cite[Theorem~1]{Alo67}, the other assertions of the theorem follow.  
\end{proof}

\begin{remark}
The proof also shows that the Scott closure operator of a chain coincides with its Dedekind--Mac Neille closure operator. 
\end{remark}



Following e.g.\ Mao and Xu \cite[Definition~3.2]{Mao06}, a poset is called \textit{hypercontinuous} if, for all $x$, the set $\{ y : y \prec x \}$ is directed and admits $x$ as supremum, where $\prec$ is the binary relation defined by $y \prec x \Leftrightarrow x \in (\uparrow\!\! y)^{o}$, the interior of $\uparrow\!\! y$ being taken in the upper topology.  

Mao and Xu \cite[Proposition~3.5]{Mao06} showed that, for a continuous poset, hypercontinuity is equivalent to the coincidence of the upper and the Scott topology (see also \cite[Theorem~VII-3.4]{Gierz03} for the case of continuous lattices), which happens to be satisfied for chains as asserted by the previous theorem. Thus, we have the following statement. 

\begin{corollary}
Every chain is a hypercontinuous poset. 
\end{corollary}


We now come to a very important result on the intrinsic topology of chains. 

\begin{theorem}
A chain equipped with its intrinsic topology is a completely normal (= hereditarily normal) space. 
\end{theorem}

\begin{proof}
See e.g.\ Al\'o and Frink \cite{Alo67} and Al\`o \cite[Theorem~1]{Alo72}. 
\end{proof}

\begin{remark}
Steen \cite{Steen70} proved the stronger result that every chain is a hereditarily collectionwise normal space. 
\end{remark}

As a corollary, we have that every chain is Tychonoff, i.e.\ a completely regular Hausdorff space (recall every normal $T_1$ space is Tychonoff). 
This fact can also be seen by \cite[Proposition~IV-3.21]{Gierz03}, which asserts that every chain allows an embedding into a cube (that is, a lattice $[0, 1]^X$) that preserves all existing sups and infs. 
Ern\'e \cite{Erne91c} proved a stronger result. Define a map $f$ between posets $P$ and $Q$ to be \textit{cut-stable} if, for all $A \subset P$, $f(A^{\uparrow})^{\downarrow} = f(A)^{\uparrow\downarrow}$ and $f(A^{\downarrow})^{\uparrow} = f(A)^{\downarrow\uparrow}$. Then, by \cite[Corollary~4.4]{Erne91c}, a chain admits a cut-stable embedding in a power-set, hence is Tychonoff. Note that every cut-continuous map is continuous with respect to interval topologies. 
The following theorem is a bit more precise. 




\begin{theorem}\label{thm:scr}
A chain $P$ equipped with its intrinsic topology is a strictly completely regular pospace, i.e.\  
\begin{enumerate}
	\item\label{thm:scr1} $P$ is strictly locally order-convex, i.e.\ it has a basis of open order-convex neighbourhoods, 
	\item\label{thm:scr2} for all closed lower (resp.\ upper) subset $A$ and $x \notin A$, there exists some continuous order-preserving map $f : P \rightarrow [0,1]$ such that $f(A) = \{0\}$ and $f(x) = 1$ (resp.\ $f(A) = \{1\}$ and $f(x)=0$). 
\end{enumerate}
\end{theorem}

\begin{proof}
We use Xu's theorem \cite[Theorem~3.3]{Xu00}, which applies if every lower subset of $P$ that is closed in the intrinsic (Lawson) topology is also closed in the lower topology. But this fact comes from \cite[Proposition~III-1.6]{Gierz03}, which says that if $G = \downarrow\!\! G$, then $G$ is dually Lawson open (i.e., open in the intrinsic topology) iff $G$ is dually Scott open (i.e., open in the lower topology).  
This result is also a consequence of Künzi \cite[Corollary~4]{Kuenzi90}, since every chain is a topological lattice. 
\end{proof}

\begin{remark}
For an example of a completey regular pospace that is not strictly completely regular, see \cite{Kuenzi90}. 
\end{remark}



That every chain be locally order-convex (Item~(\ref{thm:scr1}) of Theorem~\ref{thm:scr}) has been known since Al\`o and Frink, see \cite[Theorem~3]{Alo67}. This latter result and \cite[Theorem~1, page 36]{Gaal64} are gathered in the next theorem. 

\begin{theorem}
Every open subset of a chain (resp.\ a complete chain) is the union, in a unique way, of maximal disjoint open order-convex subsets (resp.\ disjoint open intervals). 
\end{theorem}

%




\bibliographystyle{plain}

\def\cprime{$'$} \def\cprime{$'$} \def\cprime{$'$} \def\cprime{$'$}
  \def\ocirc#1{\ifmmode\setbox0=\hbox{$#1$}\dimen0=\ht0 \advance\dimen0
  by1pt\rlap{\hbox to\wd0{\hss\raise\dimen0
  \hbox{\hskip.2em$\scriptscriptstyle\circ$}\hss}}#1\else {\accent"17 #1}\fi}
  \def\ocirc#1{\ifmmode\setbox0=\hbox{$#1$}\dimen0=\ht0 \advance\dimen0
  by1pt\rlap{\hbox to\wd0{\hss\raise\dimen0
  \hbox{\hskip.2em$\scriptscriptstyle\circ$}\hss}}#1\else {\accent"17 #1}\fi}

\end{document}